\documentclass[a4paper,reqno]{amsart}
\usepackage[utf8]{inputenc}
\usepackage[english]{babel}
\usepackage{amsmath,amssymb,amsthm,exscale,amsopn,mathtools,color,cite,enumitem}
\usepackage{graphicx}
\usepackage{color, xcolor, colortbl, tcolorbox}
\usepackage{tikz,tikz-cd,pgfplots}
\pgfplotsset{compat=1.18} 
\usepackage[colorlinks=true,pdftex,unicode=true,linktocpage,bookmarksopen,hypertexnames=false]{hyperref}
\usetikzlibrary{knots}

\usepackage{braids}
\usetikzlibrary{braids}
\usetikzlibrary{positioning, automata, arrows.meta}

\usepackage{soul}

\newtheorem{lem}{Lemma}[section]
\newtheorem{prop}[lem]{Proposition}
\newtheorem{thm}[lem]{Theorem}
\newtheorem{cor}[lem]{Corollary}
\newtheorem{defn}[lem]{Definition}

\theoremstyle{definition}
\newtheorem{ex}[lem]{Example}
\newtheorem{rem}[lem]{Remark}

\newtheorem*{convention}{Convention}

\DeclareMathOperator{\id}{id}

\DeclareMathOperator{\Inj}{Inj}

\DeclareMathOperator{\Ret}{Ret}

\DeclareMathOperator{\Soc}{Soc}

\DeclareMathOperator{\Sym}{Sym}

\DeclareMathOperator{\T}{\mathcal{S}}

\newcommand{\perm}{\mathcal{S}}

\title[Cabling of non-involutive solutions]{On the cabling of non-involutive set-theoretic solutions of the Yang--Baxter equation}
\author{I. Colazzo and A. Van Antwerpen}
\date{}

\address[I. Colazzo (ORCID: 0000-0002-2713-0409)]{School of Mathematics, University of Leeds, Leeds LS2 9JT, UK}
\email{I.Colazzo@leeds.ac.uk}

\address[A. Van Antwerpen (ORCID: 0000-0001-7619-6298)]
{Department of Mathematics and Statistics, National University of Ireland - Maynooth, Maynooth, Ireland}
\email{Arne.VanAntwerpen@ugent.be}


\begin{document}

\begin{abstract}

We extend the cabling method by Lebed, Ramírez and Vendramin from involutive to bijective non-degenerate set-theoretic solutions of the Yang--Baxter equation by working in the Yang--Baxter monoid $M(X,r)$ rather than the group $G(X,r)$. This shift in approach overcomes the obstruction that, for non-involutive solutions, the canonical map from $X$ to the Yang--Baxter group $G(X,r)$ need not be injective and yields a well-defined cabling. We prove that cabling is functorial on biquandles and that the diagonal map transforms as $q\mapsto q^k$. We also show that decomposability is preserved by injectivization and by passing to the associated biquandle, allowing us to work within that class without loss of generality. This leads to criteria for (in)decomposability. As an application, we obtain that square-free solutions with nilpotent derived monoid are decomposable.
\end{abstract}

\maketitle

\section{Introduction}

The Yang–Baxter equation has been widely studied for over six decades for its central role across research areas spanning from pure mathematics (e.g., knot theory and quantum groups) to physics (e.g., statistical mechanics and quantum field theory). Originally arising in statistical mechanics and quantum field theory, the Yang–Baxter equation has deep connections with algebra \cite{MR1321145,MR1183474}, geometry, knot theory \cite{MR908150,MR638121}, and quantum computing \cite{MR1943131}, among other fields.

In 1992, Drinfeld suggested studying the set-theoretic version of the Yang--Baxter equation. Concretely, a \emph{set-theoretic solution} is a pair $(X,r)$ where $X$ is a non-empty set and $r:X\times X \rightarrow X\times X$ is a map such that on $X^3$ one has that $$(r\times \id_X)(\id_X \times r)(r \times \id_X) = (\id_X \times r)(r \times \id_X)(\id_X \times r).$$ 
Early significant contributions to understanding set-theoretic solutions of the Yang--Baxter equation were made in \cite{MR1637256,MR1722951,LYZset}. These pioneering works systematically explored the combinatorial aspects of solutions to the Yang--Baxter equation, revealing connections between solutions, bijective 1-cocycles, monoids of $I$-type, group actions, and braiding operators. In knot theory, set-theoretic solutions to the Yang–Baxter equation emerge quite naturally, leading to the development of various algebraic structures such as quandles, racks, biquandles, and biracks.

Let us denote $r(x,y)=(\lambda_x(y),\rho_y(x))$, where for all $x\in X$, $\lambda_x, \rho_x\colon X \to X$ are maps. 
An important and widely studied class of set-theoretic solutions is that of bijective non-degenerate, see for instance \cite{MR2132760,MR3177933,MR3974961, MR4621959,MR3861714}.
A set-theoretic solution $(X,r)$ is called \emph{bijective non-degenerate} if $r$ is a bijective map and, for any $x \in X$, the maps $\lambda_x$ and $\rho_x$ are bijective. If in addition, for all  $x \in X$ one has that $r(x,x)=(x,x)$, then $(X,r)$ is said to be \emph{square-free}.

Various algebraic structures have been introduced to produce and study set-theoretic solutions of the Yang–Baxter equation. Among them, skew braces --- introduced in \cite{MR2278047,MR3647970} --- stand out as an ideal algebraic framework for studying bijective non-degenerate solutions. Since their introduction, skew braces have found applications across several areas of mathematics, including connections to triply factorizable groups \cite{MR4427114}, Garside theory \cite{MR3374524}, ring theory \cite{colazzo2023structure}, pre-Lie algebras \cite{MR4484785}, Rota--Baxter operators \cite{MR4370524}, and Hopf--Galois structures \cite{MR3763907}.

Recall that a triple $(B,+,\circ)$ is called a \emph{skew brace}, if $(B,+)$ and $(B,\circ)$ are groups such that for any $a,b,c \in B$, one has that $$a \circ (b+c) = (a\circ b) - a + (a\circ c),$$ where $-a$ denotes the inverse of $a$ in $(B,+)$. Any skew brace $(B,+,\circ)$ determines a set-theoretic solution $(B,r)$ given by $$r(a,b)=(-a+a\circ b, (a'+b)'\circ b),$$ where $a'$ denotes the inverse of $a$ in the group $(B,\circ)$.
Conversely, given a bijective non-degenerate set-theoretic solution to the Yang--Baxter equation $(X,r)$ two important skew braces can be associated to $(X,r)$: the \emph{structure (or Yang--Baxter) skew brace} $G(X,r)$ and the \emph{permutation skew brace} $\mathcal{G}(X,r)$ (see Section~\ref{sec:prelim} for definitions).

Currently a complete classification of all solutions is a far-reaching goal. One approach is to identify classes of solutions that serve as building blocks. Two such classes are indecomposable and simple solutions.

A bijective non-degenerate solution $(X,r)$ is said to be \emph{decomposable} if there is a partition $X= Y\cup Z$ with $Y,Z\neq \emptyset$ of $X$ such that $r(Y\times Y)\subseteq Y\times Y$ and $r(Z\times Z)\subseteq Z\times Z$. Otherwise $(X,r)$ is \emph{indecomposable}.

One of the first results proving decomposability for a class of non-degenerate solutions is due to Rump, who solved a conjecture of Gateva-Ivanova: in \cite{MR2132760} he proved that every involutive square-free solution is decomposable.

In \cite{MR4514462}, Ramírez and Vendramin associated to an involutive solution $(X,r)$ a permutation $q:X\to X$, given by $q(x)=\lambda_x^{-1}(x)$, which controls decomposability. More generally, Camp-Mora and Sastriques \cite{MR4565655} provided a criterion for the decomposability of involutive solutions. Following this path, Lebed, Ramírez, and Vendramin \cite{cablingorig} defined a cabling construction for involutive solutions and provided a brace-theoretic framework for the results of \cite{MR4565655}, unifying the numerical criteria available at the time and yielding new ones.

A bijective non-degenerate solution $(X,r)$ is said to be \emph{simple} if  any epimorphic image of $(X,r)$ 
is either isomorphic to $(X,r)$ or to a solution $(Y,s)$ with $Y$ a singleton.
The definition of simple solution traces back to \cite{MR1994219}, which refined the classification of simple racks by giving a concrete characterization of the underlying groups.  For involutive non-degenerate
solutions, in \cite{MR3437282}, Vendramin introduced a notion of simple involutive non-degenerate solutions. This coincides with the above definition only within the class of finite indecomposable non-degenerate solutions. Further developments include families of simple involutive solutions constructed by Cedó and Okniński \cite{MR4300920} and a group-theoretic treatment of simplicity for quandles due to Joyce \cite{MR0682881}. More recently, characterizations of simplicity in the skew left brace framework were obtained for involutive non-degenerate solutions in \cite{Cast22} and for bijective non-degenerate solutions in \cite{CJKAV24x}.

In this paper, we extend the cabling method to bijective non-degenerate solutions, with applications to indecomposable and simple solutions. There is no trivial extension to the non-involutive case; two main issues arise. First, although the canonical map $i: X\to G(X,r)$ is always injective for involutive solutions, it need not be injective for non-involutive ones. 
Moreover, we consider the biquandle associated to a bijective non-degenerate solution (cf. \cite[Definition 1.2]{MR3974961}) which is an intermediate solution sitting in between the solution and its injectivization. In particular, any involutive solution is clearly also a biquandle.  
We overcome this first obstacle by proving that indecomposability can be read off from the injectivization, Theorem~\ref{injdecomp:thm}, or from the associated biquandle, Theorem~\ref{thm:bddecomp}. The latter is a crucial result since our techniques require us to work with biquandles. 

The other problem is that for an involutive non-degenerate solution $(X,r)$ and positive integer $k$ the canonical map $i_k:X\rightarrow G(X,r)$ with $i_k(x)=kx$ is an embedding. However, this may fail even for injective solutions.
Consequently, it is unclear how to pull back the subsolution of $(G(X,r),r_G)$ on these subsets to a well-defined solution on $X$. 
We resolve this by working with the canonical embedding into the Yang---Baxter monoid $M(X,r)$, the monoid analogue of the structure skew brace, see Proposition~\ref{prop:cabledexist}. 
Finally, we prove a further structural property of cabling: it preserves morphisms of solutions, cf. Proposition~\ref{cablingendofunctor:pro}. As an application, simplicity and indecomposability are preserved under suitable cabling, see Proposition~\ref{cablingsimpleindec:pro}. 

The structure of the paper is as follows. 
Section~\ref{sec:prelim} reviews the necessary background on bijective non-degenerate solutions, biquandles, racks and quandles and the relations among them. We also recall the definitions of the Yang--Baxter group and monoid, the cancellative congruence in the latter and the definition of the Dehornoy class of a solution.

In Section~\ref{sec:dec}, we show that injectivization preserves the existence of certain homomorphisms. In particular, we show that decomposability is preserved under injectivization. Moreover, we prove that decomposability is invariant under taking the biquandle associated to a solution. 

In the final section we apply our methods to extend the theorems of Rump and of Camp-Mora and Sastriques to bijective non-degenerate set-theoretic solutions. We first show, under a mild assumption, that a square-free solution is decomposable. We then establish relations between a solution and its cabling; the main result is Theorem~\ref{thm:diagonal}, which relates the diagonal map of a solution to that of its cabling and, in the spirit of \cite[Section~4]{cablingorig}, yields numerical criteria for indecomposability.

\section{Preliminaries}\label{sec:prelim}

\begin{convention}
    In this paper, we abbreviate ``finite bijective non-degenerate solution" to ``\emph{solution}". 
\end{convention}

In this section we recall the necessary background on racks and biquandles. Moreover, we recall the structure of the Yang--Baxter monoid and its cancellative congruence.

A particular class of solutions stem from racks introduced by Joyce in \cite{MR638121} in the context of coloring knots. 
\begin{defn}
A \emph{rack} is a tuple $ (S, \triangleleft)$ with $S$ a non-empty set, $\triangleleft: S \times S \rightarrow S$ an operation such that 
\begin{enumerate}
    \item for any $x \in S$, one has that $\cdot\triangleleft x:S \rightarrow S$ is a bijection,
    \item for any $x,y,z \in S$, it holds that $ (x \triangleleft y)\triangleleft z = (x \triangleleft z)\triangleleft (y \triangleleft z)$.
\end{enumerate}
If furthermore one has that $x \triangleleft x = x$ for all $x \in S$, then $(S,\triangleleft)$ is called a \emph{quandle}. 
\end{defn}
Every rack $(S,\triangleleft)$ gives rise to a solution $(S,r)$, via $$ r(x,y)=(y,x\triangleleft y).$$ 
In the following we will use rack (or quandle) also to refer to the associated solution. Note that a quandle is precisely a square-free rack.
Conversely, it was shown in \cite{MR1809284} that every solution $(X,r)$ has an associated rack $(X,\triangleleft)$, where 
$$x \triangleleft y = \lambda_y\left( \rho_{\lambda_x^{-1}(y)}(x)\right).$$ 
This is called the (left) \emph{derived rack} of $(X,r)$; the corresponding solution is the \emph{derived solution} of $(X,r)$. Its meaning, with respect to $(X,r)$, is illustrated in Figure~\ref{square:fig}, where crossing two strands represents applying $r$.

\begin{figure}[h!]
    \centering
    \begin{tikzpicture}
\braid[
number of strands=2,
height=1cm,
gap=0.1,
nudge factor=.001,
yscale=1] 
(braid) s_1 s_1;
\node (x) at ([xshift = -.7 cm, yshift = -1.2 cm] braid) {$x$};
\node (y) at ([xshift = -.7 cm] braid) {$y$};
\node (l) at ([xshift = -1 cm,  yshift = 1.2 cm] braid) {$x \triangleleft y$};
\end{tikzpicture}
    \caption{Pictorial representation of the derived rack $(X, \triangleleft)$.}
    \label{square:fig}
\end{figure}
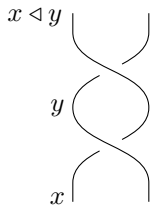

For knot colorings, the solutions of particular interest are \emph{biquandles}, as these are precisely the solutions that satisfy the Reidemeister~I move. Biquandles will play a major role in the remainder of the paper.

\begin{defn}
    Let $(X,r)$ be a solution. Then, $(X,r)$ is called a \emph{biquandle}, if the derived rack of $(X,r)$ is a quandle.
\end{defn}

In \cite{MR3974961}, it is noted that if the (left) derived rack of a solution $(X,r)$ is a quandle, then so is the right derived rack. For completeness' sake, we provide a short proof.

\begin{prop}
    Let $(X,r)$ be a biquandle. Then, the right derived rack of $(X,r)$ is a quandle.
\end{prop}
\begin{proof}
    Let $x \in X$ then, as $(X,r)$ is left non-degenerate, there exist unique $t, w \in X$ such that $r(x,t)=(x,w)$. However, as the derived rack of $(X,r)$ is a quandle, it follows that there exists a $z \in X$ such that $r^2(x,t)=r(x,w)=(x,z)$. Hence, it follows that $t=w=z$. Moreover, as $x = \rho_t^{-1}(t)$, it follows that the $x$ for which $r(x,t)=(x,t)$ is unique. This shows that for every $y \in X$, there exists an $x$ such that $r(x,y)=(x,y)$, which shows that the right derived rack is a quandle, showing the result.
\end{proof}

Now, we recall the Yang--Baxter monoid of a set-theoretic solution and its algebraic structure.

Let $(X,r)$ be a solution. Denote for all $x,y \in X$, $r(x,y) =(\lambda_x(y),\rho_y(x))$. Then, its associated Yang--Baxter monoid is the monoid $$ M(X,r) = \left< x \mid x\circ y = u\circ v \textnormal{, if } r(x,y)=(u,v)\right>.$$
As the defining relations of $M(X,r)$ are homogeneous in the length of the words, the canonical map from $X$ to the generators of $M(X,r)$ is injective, by abuse of notation we will also denote these generators as $X$. Moreover, the Yang--Baxter monoid $M(X,r)$ carries an additional additive structure, as noted in \cite{MR3974961}, similar to the skew brace structure of $G(X,r)$. Denote $A(X,r)$ the monoid with presentation $$ A(X,r) = \left< x \in X \mid x\oplus y = y\oplus\sigma_y(x) \right>,$$ where $\sigma_y(x) =x\triangleleft y= \lambda_{y}(\rho_{\lambda_{x}^{-1}(y)}(x))$. Hence, it corresponds to the Yang--Baxter monoid of the derived solution of $(X,r)$. 

It was shown by Gateva-Ivanova and Majid \cite{MR2383056} that the solution $(X,r)$ extends to a solution $(M(X,r),r_M)$ such that $r_M|_{X \times X} = r$. We denote for $a,b \in M(X,r)$, the map $r_M(a,b)= (\lambda_a(b),\rho_b(a))$. In \cite{MR4278757} it was shown that the map $\lambda: M(X,r) \rightarrow \textup{Aut}(A(X,r))$ is an action and the natural map $\id_X: X \rightarrow X$ lifts to a bijective $1$-cocycle $\pi$ from $M(X,r)$ to $A(X,r)$ for the $\lambda$-action \cite{MR3974961}. This allows to transport the semigroup structure of $A(X,r)$ to $M(X,r)$. For any $a,b \in M(X,r)$, we denote $a+b = \pi^{-1}(\pi(a)\oplus \pi(b))$ and denote $\sigma_a(b)= \pi^{-1}\sigma_{\pi(a)}(\pi(b))$. To avoid confusion, we denote the set $M(X,r)$ by $M$. Hence, the monoid $M(X,r)=(M,\circ)$ and $A(X,r) \cong (M,+)$. Note that one has, for $a,b \in M$ by the $1$-cocycle condition that $a\circ b = a + \lambda_a(b)$ and $a+b = b+\sigma_b(a)$.

Let $(X,r)$ be a solution, let $\sigma_y(x)=x\triangleleft y=\lambda_y(\rho_{\lambda^{-1}_x(y)}(x))$, with analogous presentations one can define the Yang--Baxter group $G(X,r)$ and the derived Yang--Baxter group $A_g(X,r)$.
The canonical map $i:X\to G(X,r)$ (sending generators to generators) need not be injective. Solutions for which $i$ is injective are called \emph{injective}. Moreover, $(X,r)$ induces a solution on the image $i(X)\subseteq G(X,r)$; this is an injective solution, called the \emph{injectivization} of $(X,r)$ and denoted by $Inj(X,r)$. The potential loss of information in passing to $G(X,r)$, when $i$ is not injective, is our main reason for working with the Yang--Baxter monoid instead.

However, it is known that the Yang--Baxter group can be endowed with a skew brace structure see \cite{MR3647970}.
One can consider a finite quotient of the Yang--Baxter group, the permutation group $\mathcal{G}(X,r)=\langle(\lambda_x,\rho_x^{-1})\mid x\in X\rangle$, which itself carries a skew brace structure, see for instance \cite[Theorem 3.11]{MR3835326}, and plays an important role in the study of the solution $(X,r)$.

Consider the group $\langle\sigma_x^{-1}\mid x\in X\rangle = \{\sigma_a^ {-1}\mid a \in A_g\}$ equipped with the structure of a trivial skew brace, i.e. $\sigma_a^{-1}+\sigma_b^{-1}=\sigma_a^{-1}\circ\sigma_b^{-1} = \sigma_{a+b}^{-1}$. Consider now the semidirect product of skew braces $$\perm=\perm(X,r)=\langle \sigma_x^{-1} \mid x \in X\rangle \rtimes \mathcal{G}(X,r)=\{(\sigma_a^{-1},\lambda_b,\rho_b^{-1})\mid a,b\in A_g\}$$ where $\mathcal{G}$ acts on $\langle\sigma_x^{-1}\mid x\in X\rangle$ via
\begin{align*}
    (\lambda_a,\rho^{-1}_a)\cdot \sigma_b^{-1}=\sigma_{\lambda_a(b)}^{-1}.
\end{align*}
Hence, the multiplicative group of $\perm$ is given by
\begin{align*}
    (\sigma_a^{-1},\lambda_b,\rho_b^{-1})\circ(\sigma_{a'}^{-1},\lambda_{b'},\rho_{b'}^{-1}) 
    = (\sigma_{a+\lambda_b(a')}^{-1},\lambda_{b\circ b'},\rho_{b\circ b'}^{-1}). 
\end{align*}

Then $\mathcal{S}(X,r)$ has a natural action on $X$ given by 
$$ (\sigma_a^{-1},\lambda_b,\rho_b^{-1})\cdot x = \sigma_a^{-1}\lambda_b(x),$$
with $x \in X$ and $(\sigma_a^{-1},\lambda_b,\rho_b^{-1}) \in \mathcal{S}(X,r)$

Finally, to relate a solution to its injectivization, it is crucial to understand the \emph{cancellative congruence} of $M(X,r)$. Indeed, two elements of $X$ are identified in $G(X,r)$ precisely when they are equivalent as a pair under this congruence on $M(X,r)$. To recall the cancellative congruence (determined in \cite{colazzo2023structure}), we first need the Dehornoy class. It was shown in \cite{jespers2019structure,MR3974961} that for any solution $(X,r)$ there exists a positive integer $d$ such that $\sigma_{dx}=\id_M=\lambda_{dx}$ for all $x\in X$, where $dx$ denotes the $d$-multiple in $(M,+)$. 
The least such $d$ is called the \emph{Dehornoy class} of $(X,r)$. Knowing $d$ allows one to describe explicitly the cancellative congruence of $M(X,r)$.

\begin{prop}\label{cancellativecongruence:pro}
    Denote $X=\left\lbrace x_1,...,x_n\right\rbrace$. Then, the cancellative congruence for $M(X,r)$ is given by 
 $$ \eta_M = \left\lbrace (a,b) \mid a+lz=b+lz\right\rbrace,$$ where $l$ is a particular positive integer depending on the solution, $z=dx_1+...+dx_n$ and $d$ is the Dehornoy class of $(X,r)$. 
\end{prop}  

\section{Decomposability and injectiveness}\label{sec:dec}

In this section we show that a solution is decomposable if and only if its injectivization is decomposable.
One can also extend the notion to $n$-decomposability (Definition~\ref{defn:ndecomposable}) and prove that the maximal number of factors in a decomposition is preserved under injectivization.

\begin{defn}\label{defn:ndecomposable}
   Let $(X,r)$ be a solution and $n$ be a positive integer. Then, $(X,r)$ is said to be \emph{$n$-decomposable} if there exist $n$ subsolutions $(X_1,r_1),\ldots,(X_n,r_n)$ of $(X,r)$ such that for any $i,j \in \left\lbrace 1,\ldots,n\right\rbrace$ one has that $r(X_i,X_j)=(X_j,X_i)$ and the sets $\left\lbrace X_1,\ldots,X_n\right\rbrace$ form a partition of $X$.
\end{defn}

In contrast to $2$-decomposability, for $n$-decomposability ($n>2$) the interaction of the subsolutions under $r$ is not implicit in a partition of $(X,r)$ into subsolutions. For instance, if $(X,r)$ is square-free with $|X|=n>2$ (e.g. the associated solution of a quandle of size $n$), then
$\{\{x\}\mid x\in X\}$ is a partition into subsolutions. This yields an $|X|$-decomposition if and only if $r$ is the twist map $\tau(x,y)=(y,x)$. 

Clearly, for any positive integer $n\geq 2$, an $n+1$-decomposable solution is $n$-decomposable.
Moreover, the concept of $n$-decomposability can be translated into a property of homomorphic images of solutions, i.e. the following characterization holds.

\begin{prop}
   Let $(X,r)$ be a solution and $n$ be a positive integer. Then, $X$ is  $n$-decomposable if and only if there exists an epimorphism of solutions $f\colon (X,r) \rightarrow ( \left\lbrace 0,1, \ldots, n-1\right\rbrace, \tau)$, where $\tau$ denotes the twist map $\tau(x,y)=(y,x)$.
\end{prop}

\begin{proof}
    Let $(X,r)$ be $n$-decomposable into the subsolutions $(X_0,r_0),\ldots,(X_{n-1},r_{n-1})$. Denote $f:X \rightarrow \left\lbrace 0,1,\ldots,n-1\right\rbrace$ with $f(X_i)=\left\lbrace i\right\rbrace$, which is a morphism of $(X,r)$ to $(\left\lbrace 0,1,\ldots,n-1\right\rbrace, \tau)$.

    Let $f:(X,r)\rightarrow (\left\lbrace 0,1,\ldots,n-1\right\rbrace, \tau)$ be an epimorphism of solutions.  Then $X=\bigcup_{i=0}^{n-1}f^{-1}(i)$ is a partition of $X$, which respects the solution $r$. Hence, the result follows.
\end{proof}

Clearly, a solution is decomposable if and only if it is $2$-decomposable and the following corollary holds.

\begin{cor}\label{decomphomomorph:pro}
    Let $(X,r)$ be a solution. A solution $(X,r)$ is decomposable if and only if there exists an epimorphism of solutions $f:(X,r)\longrightarrow (\left\lbrace 0,1\right\rbrace,\tau)$, where $\tau:\left\lbrace 0,1\right\rbrace \times \left\lbrace 0,1\right\rbrace \to \left\lbrace 0,1\right\rbrace \times \left\lbrace 0,1\right\rbrace$, $\tau(x,y)=(y,x)$ denotes the flip map.
\end{cor}

If $(X,r)$ is a solution, then a characterization of indececomposability in terms of $\mathcal{S}(X,r)$ from  \cite[Proposition 6.6]{MR3957824}. Namely, $(X,r)$ is indecomposable if and only if the group $\mathcal{S}(X,r)$ acts transitively on $X$. This can be further extended as follows.

\begin{prop}
    Let $(X,r)$ be a solution. The maximal positive integer $n$, such that $(X,r)$ is $n$-decomposable, is the number of orbits of $X$ under the action of $\mathcal{S}(X,r)$.
\end{prop}
\begin{proof}
    First we show that the orbits in $X$ under $\mathcal{S}$ form a decomposition of $(X,r)$. Let $Y$ denote an orbit in $X$ under $\mathcal{S}$ and $x \in X$, then it is immediate that $\lambda_x(Y)=Y$. Moreover, let $y \in Y$, then as $Y$ is an orbit under $\mathcal{S}$ one has that $$ \rho_x(y)= \lambda^{-1}_{\lambda_y(x)}\left(\sigma_{\lambda_y(x)}(y)\right)\in Y.$$
    Hence, it is shown that $r(X\times Y) = Y\times X$ and $r(Y\times X) = X \times Y$, which shows that the orbits form a decomposition of $(X,r)$.

    Finally, let $Y_1,\ldots,Y_k$ be a $k$-decomposition of $(X,r)$. As $r(X\times Y_1)=Y_1\times X$ and $r(Y_1\times X)=X \times Y_1$, it follows for any $x \in X$ that $\lambda_x(Y_1)=Y_1$ and $\rho_x(Y_1)=Y_1$. Let $y \in Y_1$. It then follows that $$ \sigma_x(y) = \lambda_{x}(\rho_{\lambda_{y}^{-1}(x)}(y)) \in Y_1.$$ Hence, $Y_1$ is a union of orbits in $X$ under $\mathcal{S}$, which shows that $k$ is bounded by the number of orbits in $X$ under $\mathcal{S}$.
\end{proof}

The following proposition is clear, as the orbits of the action of $\perm$ on $X$ are included in the orbits of the action of $\left< \sigma_x \mid x \in X\right>$, whose action coincides with that of $\mathcal{S}(X,r_{\triangleleft})$. 

\begin{prop}
    Let $(X,r)$ be a solution. If $(X,r)$ is $n$-decomposable, then the derived solution $(X,s)$ is $n$-decomposable. 
\end{prop}

As noted above, the canonical map $i:X\rightarrow G(X,r)$ is not necessarily an embedding,  as the following example shows.

\begin{ex}
    Let $X=\left\lbrace 1,\ldots,n\right\rbrace$ and $\sigma=(1...n) \in \Sym(n)$ be the proposed $n$-cycle. Denote the map $r(x_k,x_j)=(x_j,x_{\sigma(k)})$. Then $(X,r)$ is a solution, which is not injective. Indeed, $G(X,r) \cong \mathbb{Z}$ and $|i(X)|=1$. Indeed, in $G(X,r)$ one sees that $x_kx_k=x_kx_{k+1}$, where the indices should be understood modulo $n$, which shows that in $G(X,r)$ one has the equality $i(x_1)=i(x_2)=\cdots=i(x_n)$.
\end{ex}

This led to the introduction of the injectivization of a solution in \cite{MR1809284}.

\begin{defn}
    Let $(X,r)$ be a solution. Denote $i:X \rightarrow G(X,r)$. Then, the injectivization $\Inj(X,r)$ of $(X,r)$ is the subsolution $(i(X),r_{G}|_{i(X)\times i(X)})$ of $(G(X,r),r_G)$.
\end{defn}

Note that, as $G(X,r) \cong G(\Inj(X,r))$, the solution $\Inj(X,r)$ is injective. Implicitely in  \cite[Remark 5.6]{MR4278757}, it was noticed that the injectivization of a solution $(X,r)$ is an intermediate factor between $(X,r)$ and $\Ret(X,r)$, which entails, in particular, that irretractable solutions are injective. Moreover, as the cancellative congruence of $M(X,r)$ is explicitly described in \cite{colazzo2023structure}, one can determine precisely when two elements $x,y \in X$ get identified under the canonical map $i$.
 
The following theorem shows that $n$-decomposability is invariant under injectivization.
 
\begin{thm}\label{injdecomp:thm}
Let $(X,r)$ be a solution. Then $(X,r)$ is $n$-decomposable if and only if $\textup{Inj}(X,r)$ is $n$-decomposable. 

In particular, $(X,r)$ is decomposable if and only if $\Inj(X,r)$ is decomposable.
\end{thm}
\begin{proof}
The sufficiency is clear. Suppose that $(X,r)$ is $n$-decomposable into the $n$ subsolutions $X_k$. Let $i:X\longrightarrow G(X,r)$ be the canonical map. Then, if $i(X_k)\cap i(X_j)=\emptyset$ for $i \neq j$, then $i(X_1),\ldots,i(X_n)$ is an $n$-decomposition of $\Inj(X,r)$. Thus, it rests to show that $i(X_k)\cap i(X_j) =\emptyset $. Assume the contrary, there is a $y \in X_k$ and $x \in X_j$ such that $i(y)=i(x)$. Then, $(y,x)\in \eta_{M}$. Thus, $$y+ lz = x +lz.$$ Note that $X_1,\ldots,X_n$ also induce an $n$-decomposition of the derived solution, which entails that the number of elements in $X_k$ and $X_j$ in a word in $A(X,r)$ is invariant under applying its defining relations. However, in $y+lz$ there are precisely $ld|X_k|+1$ elements of $X_k$ and $ld|X_j|$ of $X_j$, which differs from the $ld|X_k|$ elements of $X_k$ and $ld|X_j|+1$ elements of $X_j$ in $x+lz$, producing a contradiction. Hence, $i(X_j)\cap i(X_k) = \emptyset$, which shows the result.
\end{proof}

In particular, the previous shows that the injectization only identifies elements in the same subsolution. 

Note that this can not be extended to a multi-stage setting. Indeed, if $(X,r)$ has a $k$-decomposition into $X_1,\ldots,X_k$, and $X_1$ has a further decomposition into subsolutions $X_{11}$ and $X_{12}$, the injectivization of $(X,r)$ may identify elements of $X_{11}$ and $X_{12}$ as shown in the following example.
\begin{ex}
    Consider $\left\lbrace 0,1,2 \right\rbrace$ with $\sigma_0=\sigma_1=\id$ and $\sigma_2=(01)$. Then, $r(x,y)=(y,\sigma_y(x))$ is a solution. Moreover, $\left\lbrace 0,1,2\right\rbrace$ is $2$-decomposable into $Y=\left\lbrace 0,1\right\rbrace$ and $\left\lbrace 2 \right\rbrace$. Note that $Y$ is the twist solution and hence decomposable into $\left\lbrace 0\right\rbrace$ and $\left\lbrace 1 \right\rbrace$. However, in $G(X,r)$ one has that $$ i(0) + i(2) = i(2) + i(1)= i(1) + i(2),$$ which shows that in $G(X,r)$ one has that $i(0)=i(1)$.
\end{ex}

The biquandle associated (cf. Definition~\ref{def:biquandleas}) to a solution, as defined by Lebed and Vendramin in \cite[Definition 1.2]{MR3974961}, it is an intermediate solution between a solution and its injectivization. Recall that a biquandle is a solution, where the derived solutions are quandles. 

\begin{defn}\label{def:biquandleas}
    Let $(X,r)$ be a solution. Then, the \emph{biquandle $\textup{BQ}(X,r)$ associated to $(X,r)$} is the induced solution on $X/\sim$, where $\sim $ is the smallest equivalence relation such that $x \sim \lambda_x(y)$, if $y=\rho_y(x)$. 
\end{defn}

\begin{rem}\label{factormorphism:rem}
As noted in \cite{MR3974961} the natural morphism $(X,r) \rightarrow \textup{Inj}(X,r)$ factors into the natural morphisms $(X,r) \rightarrow \textup{BQ}(X,r)\rightarrow \textup{Inj}(X,r)$.
\end{rem}

This will allow us to show that decomposability is invariant under passage to the associated biquandle. This is crucial for the remainder of the paper, since our techniques require working with biquandles; the next result shows that no generality is lost by doing so.

\begin{thm}\label{thm:bddecomp}
    Let $(X,r)$ be a solution. Then, $(X,r)$ is $n$-decomposable if and only if $\textup{BQ}(X,r)$ is $n$-decomposable.

    In particular, $(X,r)$ is decomposable if and only if $\textup{BQ}(X,r)$ is decomposable.
\end{thm}
\begin{proof}
    Note that due to Remark~\ref{factormorphism:rem} the injectivization of $(X,r)$ and $\textup{BQ}(X,r)$ coincide in $\textup{Inj}(X,r)$. Hence, the result follows from Theorem~\ref{injdecomp:thm}.
\end{proof}

\section{Cabling: functoriality and other properties}\label{sec:cabling}

In \cite{cablingorig} Lebed, Ramírez, and Vendramin define cabling operations for involutive non-degenerate solutions of the Yang--Baxter equation, using particular subsolutions of $G(X,r)$. The latter provides a unifying framework to deal with known indecomposability results of involutive solutions. In this section, we extend the cabling technique to bijective non-degenerate solutions.
A key difference is that, in place of subsolutions of $G(X,r)$, we work with subsolutions of the Yang--Baxter monoid $M(X,r)$. Indeed, for an involutive non-degenerate solution $(X,r)$ and positive integer $k$ the canonical map $i_k:X\rightarrow G(X,r)$ with $i(x)=kx$ is an embedding. However, this may fail even for injective non-involutive solutions.

\begin{ex}
Let $X=\left\lbrace x_1,x_2,x_3\right\rbrace$ and $\sigma_1=(23),$ $\sigma_2=(13)$ and $\sigma_3=(12)$. Then, the solution $r(x_j,x_k)=(x_k,x_{\sigma_{k}(j)})$ is injective and indecomposable. However, in $G(X,r)$ one has that $2x_1=2x_2=2x_3$.   
\end{ex}

This shows that, beyond the involutive setting, it is not clear whether one can pull back the subsolution of $(G(X,r),r_G)$ on the subsets $kX$ to a well-defined solution on $X$. We resolve this by working with the canonical embedding into the structure monoid $M(X,r)$.

\begin{prop}\label{prop:cabledexist}
Let $(X,r)$ be a solution. Then the subset $$kX=\left\lbrace (kx,\lambda_{kx}) \mid x \in X \right\rbrace \subseteq M(X,r)$$ is a subsolution $(kX,r_k)$ of $(M,r_M)$. 
\end{prop}
\begin{proof}
Note that the maps $\lambda_a$ and $\sigma_a$ are both automorphisms of $(M(X,r),+)$, hence $\lambda_a(kx)=k\lambda_a(x)$ and $\sigma_a(kx)=k\sigma_a(x)$. As $$ \rho_b(kx)=\lambda_{\lambda_{kx}(b)}^{-1}\sigma_{\lambda_{kx}(b)}(kx)=k\lambda_{\lambda_{kx}(b)}^{-1}\sigma_{\lambda_{kx}(b)}(x),$$  it follows that $kX$ gives rise to a subsolution of $M(X,r)$.
\end{proof}

\begin{defn}
    Let $(X,r)$ be a biquandle. Denote the bijection $\varphi:X\rightarrow kX$ with $\varphi(x)=kx$ and denote $r^{(k)}=(\varphi^{-1}\times \varphi^{-1})r_k(\varphi\times\varphi)$. Then $(X,r^{(k)})$ is called the \emph{$k$-cabled solution} of $(X,r)$.
\end{defn}

Note that an involutive solution $(X,r)$ is injective, and thus a biquandle, and the cabling defined in \cite{cablingorig} for involutive solutions coincides with the construction above, as $M(X,r)$ is embedded in $G(X,r)$ \cite[Corollary 1.5]{MR1637256}.

    \begin{prop}\label{cablinginjective:pro}
        Let $(X,r)$ be a biquandle. Let $k$ be a positive integer, then the $k$-cabled solution $(X,r^{(k)})$ is a biquandle.
    \end{prop}
    \begin{proof}
        The result follows from the fact that $\sigma_{kx}=\sigma_{x}^k$. Hence, $$\sigma_{kx}(kx)=k\sigma_{kx}(x)=k\sigma_{x}^k(x)=kx.$$
    \end{proof}

Crucial for using cabling to construct new solutions with interesting properties from old ones is the following proposition, which shows that cabling preserves morphisms.
    \begin{prop}\label{cablingendofunctor:pro}
    Let $f:(X,r) \rightarrow (Y,s)$ be a morphism of biquandles and $k$ a positive integer. Then, $f$ is a morphism between the cabled biquandles $(X,r^{(k)})$ and $(Y,s^{(k)})$.
    \end{prop}
    \begin{proof}
        Let $f:(X,r) \rightarrow (Y,s)$ be a morphism of biquandles, then, via the universal property of the monoids $M(X,r)$ and $A(X,r)$, the morphism $f$ extends to a morphism of semigroups $f_M:M(X,r) \longrightarrow M(Y,s)$ and morphism $f_A: A(X,r) \rightarrow A(Y,s)$, such that $\pi f_M = f_A\pi$. As $f_A(kx) = kf_A(x)$ for all $x \in X$, the map $f_M$ restricts to a morphism of solutions between $\left\lbrace \pi^{-1}(kx)\mid x \in X \right\rbrace$ and $\left\lbrace \pi^{-1}(y) \mid y \in Y\right\rbrace$. This entails that then the map $f^{(k)}:(X,r^{(k)})\rightarrow (Y,s^{(k)})$ with $f^{(k)}(x)=f(x)$ is a morphism of solutions.
    \end{proof}
    Summarizing Proposition~\ref{cablinginjective:pro} and Proposition~\ref{cablingendofunctor:pro}, cabling is a functor from the category of biquandles to itself. 

    \begin{rem}
        Note that the cabling functor is faithful. Indeed, a morphism of biquandles is a set-theoretic map, and cabling does not change this map. Hence, distinct morphisms remain distinct under cabling.
    \end{rem}
    
    The following proposition relates the cabling of a cabled solution to a cabling of the original solution.
    \begin{prop}\label{cablingofcabling:pro}
        Let $(X,r)$ be a biquandle. Let $k$ and $k'$ be positive integers, then  $$(X,(r^{(k)})^{(k')})=(X,r^{(kk')}).$$
    \end{prop}
    \begin{proof}
       Let $x,y \in X$. Then, one obtains that $$\lambda_x^{(k)^{(k')}}(y)=\lambda_{k'x}^{(k)}(y)=\lambda_{kk'x}(y)=\lambda_x^{(kk')}(y). $$ Similarly, one obtains that $\sigma_x^{(k)^{(k')}}=\sigma^{(kk')}_x$, which shows the result.
    \end{proof}
    Note that this entails that if one considers a biquandle $(X,r)$ with Dehornoy class $d$ and all its cablings, then the multiplicative monoid of $\mathbb{Z}_d$, i.e. integers modulo $d$, acts via cabling on this set, where $(X,r^{(0)})$ denotes $(X,r^{(d)})$.

    Moreover, Propositions~\ref{cablingendofunctor:pro} and \ref{cablingofcabling:pro} allow us to show that cabling preserves several classes of biquandles.

    \begin{prop}
        Let $S$ be a non-empty class of biquandles closed under cabling. Denote $T$ (resp. $T_{epi}$) the class of biquandles that have a homomorphism (resp. epimorphism) to an element in $S$, then $T$ (resp. $T_{epi}$) is closed under cabling. In particular, decomposability of biquandles is preserved under cabling. 
    \end{prop}
    \begin{proof}
        Let $(X,r)$ be a biquandle with a homomorphism to an element of $S$, say $(Y,s)$. Denote this homomorphism $\varphi:(X,r) \rightarrow (Y,s)$. Then, for any positive integer $k$ the $k$-cabling $(Y,s^{(k)})\in S$. By Proposition \ref{cablingendofunctor:pro}, the induced map $\varphi^{(k)}:(X,r^{(k)})\rightarrow (Y,s^{(k)})$ is a homomorphism of solutions, so $(X,r^{(k)})\in T$. Note that if $\varphi$ is an epimorphism, the same holds for $\varphi^{(k)}$.

        Remark that the class of decomposable biquandles coincides with the class of biquandles with an epimorphism to the solution $(\left\lbrace0,1\right\rbrace, \tau)$, where $\tau(x,y)=(y,x)$, as shown in Proposition~\ref{decomphomomorph:pro}. As this solution is invariant under cabling, the result follows.
    \end{proof}
    
    \begin{prop}\label{cablingsimpleindec:pro}
          Let $S$ be a non-empty class of biquandles closed under cabling and let $d$ be a positive integer. Denote the class of biquandles of Dehornoy class dividing $d$ that do not have a homomorphism to any element of $S$ as $W$, then $W$ is closed under $k$-cabling with $k$ a positive integer coprime to $d$. In particular, if $(X,r)$ is a simple (resp. indecomposable) biquandle of Dehornoy class $d$, then $(X,r^{(k)})$ is a simple (resp. indecomposable) biquandle if $k$ is coprime to $d$.
    \end{prop}
    \begin{proof}
        Let $(X,r)$ be a biquandle with Dehornoy class dividing $d$ that does not have a homomorphism to any element of $S$. Let $k$ be a positive integer coprime with $d$. Suppose $\varphi:(X,r^{(k)})\rightarrow (Y,s)$ is a homomorphism of solutions with $(Y,s) \in S$. Denote $l$ a positive integer such that $kl =1$ modulo $d$. Then, by Proposition \ref{cablingofcabling:pro}, one has that $\varphi^{(l)}: (X,r^{kl})\rightarrow (Y,s^{(l)})$ is the $l$-cabled homomorphism of solutions with $(Y,s^{(l)})\in S$. However, $(X,r^{(kl)})=(X,r)$ which contradicts that $(X,r) \in W$.

        Note that a  simple (resp. indecomposable) biquandle of size $n$ can be characterized as a biquandle with no epimorphism to an biquandle of size between $2$ and $n-1$ (resp. to the solution $(\{0,1\},\tau)$, where $\tau(x,y)=(y,x)$, by Proposition \ref{decomphomomorph:pro}).
    \end{proof}

    Note that a simple solution  is either isomorphic to $$r:\mathbb{Z}_p^2\to \mathbb{Z}_p^2, (x,y)\mapsto(y+a,x+b)$$ with $a,b, \in\mathbb{Z}_p$ and $(a,b)\neq (0,0)$ or injective, and hence a biquandle \cite{CJKAV24x}. Thus, the restriction to work with biquandles in Proposition~\ref{cablingsimpleindec:pro} is non-restrictive.

    Moreover, it is worth noting that the class of injective solutions is not closed under epimorphic images. An example can be constructed using \textsc{GAP} \cite{gap4}. Consider the quandle \texttt{SmallQuandle(24,21)} from the package \texttt{Rig} \cite{rig}, which is a conjugacy class of the group \texttt{SmallGroup(216,153)} from the \texttt{Small Groups Library} \cite{smallgroups}. This group has order 216 and admits an epimorphic image in the non-injective quandle \texttt{SmallQuandle(8,1)} of order 8 (see \cite[Example 4.5]{MR3974961} for an explicit description of this quandle and the proof that it is non-injective).
    
    One may also ask whether cabling with a parameter coprime to the Dehornoy class can produce non-isomorphic solutions. The following example shows that this is indeed the case.

\begin{ex}
    Let $X=\{1,2,3,4,5, 6\}$ and consider $(X,r)$ the involutive solution on $X$ given by 
    \begin{align*}
        \lambda_1&=(1, 2, 3)\\
         \lambda_2 &=(1,2,3)(4,5,6)\\
         \lambda_3&= (1,2,3)(4,6,5)
    \end{align*}
    and $\lambda_x=(4,5,6)$ for $x\in\{4,5,6\}$.
    The Dehornoy class of $(X,r)$ is $3$. The $2$-cabled solution has
        \begin{align*}
        \lambda_1&=(1,3,2)(4,6,5)\\
         \lambda_2 &=(1,3,2)(4,5,6)\\
         \lambda_3&= (1,3,2) 
    \end{align*}
    and $\lambda_x=(4,6,5)$ for $x\in\{4,5,6\}$. The $2$-cabled solution is clearly not isomorphic to the original one.
\end{ex}

Since the retract plays an important role in the study of solutions we first note that injectivity is preserved by retraction.
    \begin{lem}
        Let $(X,r)$ be an injective solution. Then the retract of $(X,r)$ is injective.
    \end{lem}
    \begin{proof}
        As $(X,r)$ is injective, it is embedded in $G(X,r)$. The structure skew brace of $\Ret(X,r)$ is a homomorphic image of $G(X,r)$ modulo an ideal of $G(X,r)$ contained in $\Soc(G(X,r))$ generated by the generators $x-y$ with $x,y \in X$ with $\lambda_x=\lambda_y$ and $\sigma_x=\sigma_y$. Hence, it is clear that the canonical embedding $\Ret(X,r)\rightarrow G(\Ret(X,r))$ is injective, showing the result.
    \end{proof}

We next show that retraction behaves well with respect to cabling.
    \begin{prop}
        Let $(X,r)$ be a biquandle. If $(X,r)$ is retractable, then every cabling of $(X,r)$ is retractable. In particular, the property of having finite multipermutation level is closed under cabling.
    \end{prop}
    \begin{proof}
        Let $x,y \in X$ be such that $\lambda_x=\lambda_y$ and $\sigma_x=\sigma_y$. Then, by induction on $n$, we have $$ \sigma_{nx}=\sigma_x\sigma_{(n-1)x} = \sigma_y\sigma_{(n-1)y}=\sigma_{ny}.$$ Thus, it remains to prove that $\lambda_{kx}=\lambda_{ky}$ for all positive integers $k$. We proceed by induction. Note that $$\lambda_{(n+1)x}= \lambda_x \lambda_{\lambda_{x}^ {-1}(nx)} =\lambda_{y}\lambda_{\lambda_y^{-1}(nx)}=\lambda_{y+nx}.$$ Utilizing the defining relation in the additive group, the latter is equal to $\lambda_{nx+\sigma_{nx}(y)}$. By the induction hypothesis, this is further equal to $$\lambda_{nx}\lambda_{\lambda_{nx}^{-1}(\sigma_{nx}(y))} = \lambda_{ny}\lambda_{\lambda_{ny}^{-1}\sigma_{ny}(y)} = \lambda_{ny+\sigma_{ny}(y)}=\lambda_{(n+1)y}. $$  Hence, we have shown that for any positive integer $k$ one has that $\lambda_{kx}=\lambda_{ky}$, which shows the result.
    \end{proof}

    \begin{cor}
        Let $(X,r)$ be a biquandle. Let $(Y,s)$ be a subsolution of $(X,r)$. Then, for any positive integer $k$ the solution $(Y,s^{(k)})$ is a subsolution of $(X,r^{(k)})$. In particular, if $k$ is coprime to the Dehornoy class of $(X,r)$ the subsolutions of $(X,r)$ and $(X,r^{(k)})$ are in bijective correspondence.
    \end{cor}
    \begin{proof}
        A subsolution of $(X,r)$ corresponds to an embedding of the solution $(Y,s)$ into $(X,r)$. By Proposition \ref{cablingendofunctor:pro}, one obtains an embedding from $(Y,s^{(k)})$ into $(X,r^{(k)})$.
    \end{proof}

    \begin{rem}
        Let $B$ be a finite skew brace and $k$ a positive integer coprime with the Dehornoy class of $B$. Then the solutions $(B,r^{(k)}_B)$ and $(B,r_B)$ are isomorphic. Indeed, the map $f:(B,r^{(k)}_B)\rightarrow (B,r_B)$ with $f(x)=kx$ gives an isomorphism.
    \end{rem}

\section{Decomposability tests via cabling}

In this section, we show, under mild assumptions, that if a solution is square-free then it has to be decomposable (Proposition~\ref{decompnilpot:prop}). For involutive solutions this was conjectured by Gateva-Ivanova (1996) and proved by Rump (2005) \cite{MR2132760}.

In Theorem~\ref{thm:diagonal} we relate the diagonal map of $(X,r)$ and of its cabled solution $(X,r^{(k)})$ we also establish a relation between the orbits in $\mathcal{S}(X,r)$ and $\mathcal{S}(X,r^{(k)})$.

In the spirit of \cite[Section~4]{cablingorig}, Theorem~\ref{thm:diagonal} yields criteria for (in)decomposability.

The following result shows that a quandle with nilpotent Yang--Baxter group has an associated decomposable solution. This result was also shown in \cite[Corollary 3.7]{Darn}. We, however, provide a different, completely group-theoretical proof of this result that will provide the foundation for the proof of Proposition~\ref{decompnilpot:prop}.

\begin{prop}\label{prop:quandledecomp}
    Let $(X,\triangleleft)$ be a quandle with $|X|>1$. If the Yang--Baxter group $G(X,\triangleleft)$ is nilpotent and not isomorphic to $\mathbb{Z}$, then the associated set-theoretic 
 solution $(X,r_{\triangleleft})$ is decomposable.
\end{prop}
\begin{proof}
    We may assume that $(X,r_{\triangleleft})$ is an injective solution (Theorem~\ref{injdecomp:thm}), i.e. the canonical map $X \rightarrow G(X,\triangleleft)$ is injective. Suppose the set-theoretic solution $(X,r_{\triangleleft})$ is indecomposable. Let us write $r_{\triangleleft}(x,y) = (y, \sigma_y(x))$. Then, we can identify the permutation group $\mathcal{G}$ with $\left<\sigma_x^{-1}\mid x \in X \right>$ and since $(X,r_{\triangleleft})$, the action of $\mathcal{G}$ on $X$ is transitive. Consider a maximal intransitive normal subgroup $N$ of $\mathcal{G}$, i.e. every normal subgroup of $\mathcal{G}$ containing $N$ acts transitively on $X$. Let $N'$ be a minimal element in the set of normal subgroups containing $N$. As $\mathcal{G}$ is a homomorphic image of the Yang--Baxter group  $G(X,\triangleleft)$, it follows that $N'/N$ is an elementary abelian $p$-group. Let $\mathcal{G}_p$ denote the Sylow $p$-subgroup of $\mathcal{G}$, then $N'\subseteq \mathcal{G}_pN$, which implies that $\mathcal{G}_p$ acts transitively on $X/N$. Let $P$ be a maximal intransitive normal subgroup of $\mathcal{G}_p$, then there exists, similar to before, a normal subgroup $P'$ containing $P$ such that $|P'/P|=p$ and $P'$ acts transitively on $X/N$. Moreover, it acts transitively on the set $Y$ of $P$-orbits of $X/N$, which implies that $|Y|=p$. Denote by $X_p$ the image of $X$ under the composition of the embedding $X\rightarrow \mathcal{G}$ and the projection on  $\mathcal{G}_p$. As the latter is generated by $X_p$, it follows that there exists an $x \in X$ and a positive integer $m$ such that $\sigma_{mx}$ permutes $Y$. As the latter is of $p$-power order, it follows that it is a $p$-cycle that permutes $Y$. However, considering the fact that $$\sigma_{mx}(x)=\sigma_x^m(x)=x,$$ it follows that $\sigma_{mx}$ leaves the orbit containing $x$ invariant, which is a contradiction.
\end{proof}

\begin{cor}
Let $(X,\triangleleft)$ be a rack. If the Yang--Baxter group $G(X,\triangleleft)$ is nilpotent and is not isomorphic to $\mathbb{Z}$, then the associated solution $(X,r_{\triangleleft})$ is decomposable.
\end{cor}
\begin{proof}
    By Theorem~\ref{injdecomp:thm}, $(X,r_{\triangleleft})$ is decomposable if and only if its injectivization is decomposable. Note that in $G(X,\triangleleft)$ one has that $$ x+x=x+\sigma_x(x),$$ which entails that $\Inj(X,r)$ is square-free. As $G(X,r_{\triangleleft})\cong G(X,\triangleleft)\cong G(\Inj(X,r))$, the latter is nilpotent. Hence, the result follows from Proposition~\ref{prop:quandledecomp}.
\end{proof}

The following example on $3$ elements shows that Theorem~\ref{prop:quandledecomp} can not be extended to square-free derived solutions with (super)solvable permutation groups.

\begin{ex}
    Consider the group $S_3$ and the solution $s(\sigma,\tau)=(\tau,\tau^{-1}\sigma\tau)$, i.e. the solution coming from the conjugation quandle on $S_3$. Denote $X=\left\lbrace (12),(23),(13)\right\rbrace$. Then, $(X,s)$ is a square-free indecomposable bijective non-degenerate solution, and $\mathcal{G}(X,s) \cong S_3$ is solvable but not nilpotent.
\end{ex}

\begin{prop}\label{decompnilpot:prop}
    Let $(X,r)$ be a finite solution. If $(X,r)$ is square-free and $A(X,r)$ is nilpotent, then $(X,r)$ is decomposable.
\end{prop}
\begin{proof}
    We may assume that $(X,r)$ is an injective solution. Suppose that $(X,r)$ is indecomposable, then $\perm(X,r)$ acts transitively on $X$. By Proposition~\ref{prop:quandledecomp}, the group $\left<\sigma_x^{-1} \mid x \in X\right>$ acts intransitively on $X$. Hence, there exists a maximal intransitive normal subgroup $N$ of $\perm(X,r)$ containing $\left<\sigma_x^{-1}\mid x \in X \right>$. Moreover, identifying $\perm(X,r)/N$ with a suitable homomorphic image of $\mathcal{G}(X,r)$, one obtains that the induced action of $\mathcal{G}(X,r)$ on $X/N$ is transitive. Denote $N'$ the image of $N$ under this identification. Let $M$ be a minimal element in the set of normal subgroups of $\mathcal{G}(X,r)$ containing $N'$. As $A(X,r)$ is nilpotent, it follows that $\mathcal{G}(X,r)$ is a solvable group. Hence, $M/N'$ is an elementary abelian $p$-group. Denote $\mathcal{G}_p=(\mathcal{G},+)_p$ for the $p$-component of $(\mathcal{G},+)$. As the latter is a strong left ideal of $\mathcal{G}$, it is a Sylow $p$-subgroup of $(\mathcal{G},\circ)$. Hence, it contains $M$ and thus $\mathcal{G}_p$ acts transitive on $X/N$. Let $P$ denote a maximal intransitive multiplicative normal subgroup of $\mathcal{G}_p$ and $P'$ a minimal element in the set of multiplicative normal subgroups of  $\mathcal{G}_p$ containing $P$, then $|P'/P|=p$. Moreover, $P'/P$ acts transitively on the set $Y$ of $P$-orbits of $X/N$. Thus, $|Y|=p$. Denote by $X_p$ the image of $X$ under the composition of the embedding $X\rightarrow (\mathcal{G},+)$ and the projection on $\mathcal{G}_p$. Then $X_p$ generates $\mathcal{G}_p$ additively. Clearly, there should exist an element $x\in X$ such that the image of $x$ in $X_p$ acts non-trivially on $Y$. Thus, there exists an $m$ positive integer such that $(\sigma_{mx},\lambda_{mx})$ acts non-trivially on $Y$ and is of $p$-power order, which implies it acts as a $p$-cycle on $Y$. Note, as  $(X,r)$ is square-free, it holds that $$ (\sigma_{mx},\lambda_{mx})*x= \sigma_{mx}^{-1}\lambda_{mx}(x)=\sigma_x^{-m}\lambda_x^m(x)=x.$$ In particular, $(\sigma_{mx},\lambda_{mx})$ leaves the $P$-orbit containing $x$ invariant, which is a contradiction.
    \end{proof}
    \begin{rem}
        Proposition~\ref{decompnilpot:prop} can be completely reformulated in terms of the solution. Indeed, it is straightforward from \cite[Theorem~2.20]{MR3957824} that $A(X,r)$ is nilpotent if and only if the derived rack is of finite multipermutation level. Hence, Proposition~\ref{decompnilpot:prop} says that a square-free solution with left derived rack of finite multipermutation level is necessarily decomposable. We also refer to \cite{MR4457900} were more results on solutions with nilpotent Yang--Baxter group are obtained.
    \end{rem}
    
The following result may be used to derive numerical conditions for a solution to be indecomposable. 
    \begin{thm}\label{thm:diagonal}
        Let $(X,r)$ be a biquandle and $k$ a positive integer. Then, the (left) diagonal map of $(X,r^{(k)})$ is $q^k$. Moreover, if $x \in X$ lies in a $\mathcal{S}(X,r)$ orbit of size $m$ and a $\mathcal{S}(X,r^{(k)})$-orbit of size $m'$. Denote $m_k$ the maximal divisor of $m$ coprime to $k$. Then $m'$ is a multiple of $m_k$.
    \end{thm}
    \begin{proof}
        Note that for any $x \in X$, we have that $kx = x \circ q(x) \circ \cdots \circ q^{k-1}(x)=x \circ ((k-1)q(x))$. Hence, $$\lambda_{kx}^{-1}(x) = \lambda_{(k-1)q(x)}^{-1}(q(x)).$$ Continuing this process, one obtains that $$ \lambda_{kx}^{-1}(x)=q^k(x),$$ which proves the first statement. 

        Let $x \in X$. The stabilizer of $x$ under the action of $\mathcal{S}(X,r)$ is denoted by $\mathcal{S}(X,r)_x$. Due to the orbit-stabilizer theorem, one has that $|\T(X,r)|=m|\T(X,r)_x|$ and $|\T(X,r^{(k)})|=m'|\T(X,r^{(k)})_x|$. Furthermore, as $\T(X,r^{(k)})$ is generated by $\sigma_{ky}$ and $(\lambda_{kz},\sigma_{kz})$ for $y,z \in X$, one sees that $\T(X,r^{(k)})$ is a subskew brace of $\T(X,r)$. Moreover, these generators are $k$-multiples of the generators of $\T(X,r)$. Hence, $|T(X,r^{(k)})|= |\T(X,r)|/\alpha$ for some $\alpha$ whose prime divisors divide $k$. One sees that the stabilizer $\T(X,r^{(k)})_x$ is a multiplicative subgroup of $\T(X,r)_x$. Thus, for some positive integer $t$, one has that $|\T(X,r)_x|/t=|\T(X,r^{(k)})_x|$. Thus, $$m'=|\T(X,r^{(k)})|/|\T(X,r^{(k)})_x|= \frac{|\T(X,r)|t}{\alpha |\T(X,r)_x|}= \frac{mt}{\alpha}.$$ Recall that $\alpha$ is a divisor of $m$, whose prime divisors only divide $k$. Hence, the latter is a multiple of $m_k$, showing the result.
    \end{proof}

As a corollary we obtain the non-involutive version of \cite[Theorem~A]{MR4565655}.
    \begin{cor}
        Let $(X,r)$ be a solution. If $A(X,r)$ is nilpotent and $\textup{gcd}(|X|,|q|)=1$, then $(X,r)$ is decomposable.
    \end{cor}
    \begin{proof}
        By Theorem~\ref{injdecomp:thm}, the solution $(X,r)$ is indecomposable if and only if its injectivization is indecomposable. Hence, we may restrict to injective solutions. Then, one notices that the diagonal map of the $|q|$-cabled solution $(X,r^{(|q|)})$ is the identity. Moreover, $A(X,r^{|q|})$ is nilpotent. Indeed, $A(X,r^{|q|})$ modulo its center is a subgroup of $\left< \sigma^{-1}_{|q|x} \mid x \in X\right>$. This, in turn, is a subgroup of $\left< \sigma_x^{-1} \mid x \in X \right>$ and isomorphic to $A(X,r)$ modulo its center. Hence, by Proposition~\ref{decompnilpot:prop}, the solution is decomposable.
    \end{proof}

The permutation $q$ splits $X$ into orbits. This induces a partition of $n =|X|$, which we call the $q$-partition of $(X,r)$.
With the same proof as in the previous corollary considering the $p^t$-cabling, for a suitable $t$, we also obtain.
\begin{cor}
    Let $(X,r)$ be an indecomposable solution $(X, r)$ with multipermutation derived solution of size $ps$, where $p \neq s$ are prime. Then its $q$-partition cannot contain a term $t$ satisfying $(p-1)s<t<ps$ and $\gcd(t,p)=1$.
\end{cor}

Analogously, replacing indecomposable involutive with indecomposable with multipermutation derived solution, one obtains the non-involutive version of Theorems~D and E in \cite{cablingorig}.
 
\section*{Acknowledgment}
We gratefully acknowledge the partial support of Fonds voor Wetenschappelijk Onderzoek (Flanders) -- Krediet voor wetenschappelijk verblijf in Vlaanderen grant V512223N (Ilaria Colazzo), and the London Mathematical Society, joint research group 32312, which enabled parts of this collaboration to take place.

\bibliographystyle{abbrv}
\bibliography{refs}
\end{document}